\documentclass{amsart}
\usepackage{amsmath, amssymb, amsthm,enumerate,booktabs}
\usepackage{hyperref}
\usepackage[margin=1in]{geometry}
\usepackage[foot]{amsaddr}
\hypersetup{citecolor=red, linkcolor=blue, colorlinks=true}

\newcommand{\scrB}{\mathcal{B}}
\newcommand{\bbF}{\mathbb{F}}
\newcommand{\rank}{\text{rank}}

\renewcommand{\mod}[1]{\text{ (mod $#1$)}}

 \numberwithin{equation}{section}
 \newtheorem{theorem}[equation]{Theorem}
 \newtheorem{cor}[equation]{Corollary}
 
 \newtheorem{lemma}[equation]{Lemma}

\newtheorem{problem}{Problem}
 
 \theoremstyle{definition}

 \newtheorem{ex}[equation]{Example}

\title
 [Delsarte Cocliques in Strongly Regular Graphs]{Ovoids of Generalized Quadrangles of Order $(q, q^2-q)$ and Delsarte Cocliques in Related Strongly Regular Graphs}
\author[]{Mohammad Adm\textsuperscript{1}}
\author{Ryan Bergen\textsuperscript{2}}
\author{Ferdinand Ihringer\textsuperscript{3}}
\author{Sam Jaques\textsuperscript{2}}
\author{Karen Meagher\textsuperscript{4}}
\author{Alison Purdy\textsuperscript{2}}
\author{Boting Yang\textsuperscript{5}}

\address{\textsuperscript{1}Department of Mathematicsand Statistics, University of Konstanz, Konstanz, Germany and Department of Mathematics and Statistics, University of Regina, Regina, Canada. Adm's research was supported by the German Academic Exchange Service (DAAD) with funds from the German Federal Ministry of Education and Research (BMBF) and the People Programme (Marie Curie Actions) of the European Union’s Seventh Framework Programme (FP7/2007-2013) under REA grant agreement no. 605728 (P.R.I.M.E. – Postdoctoral Researchers International Mobility Experience).}
\address{\textsuperscript{2}Department of Mathematics and Statistics, University of Regina.}
\address{\textsuperscript{3}Einstein Institute of Mathematics, Hebrew University of Jerusalem, Israel. Research supported by a PIMS postdoctoral fellowship while the author was a postdoctoral fellow at the University of Regina.}
\address{\textsuperscript{4}Department of Mathematics and Statistics, University of Regina. Research supported in part by an NSERC Discovery Research Grant, Application No.:RGPIN-341214-2013.}
\address{\textsuperscript{5}Department of Computer Science, University of Regina. Research supported in part by an NSERC Discovery Research Grant, Application No.:RGPIN-2-13-261290.}

\begin{document}

\begin{abstract}
  We investigate strongly regular graphs for which Hoffman's ratio
  bound and Cvetcovi\'{c}'s inertia bound are equal. This means that
  $ve^- = m^-(e^- - k)$, where $v$ is the number of vertices, $k$ is
  the regularity, $e^-$ is the smallest eigenvalue, and $m^-$ is the
  multiplicity of $e^-$.  We show that Delsarte cocliques do not exist
  for all Taylor's $2$-graphs and for point graphs of generalized
  quadrangles of order $(q,q^2-q)$ for infinitely many $q$.  For cases
  where equality may hold, we show that for nearly all parameter sets,
  there are at most two Delsarte cocliques.
\end{abstract}

\maketitle



\section{Introduction}

Hoffman's ratio bound and Cvetcovi\'{c}'s inertia bound are two of the
best known bounds for cocliques in regular graphs.  We investigate the
case where the unweighted versions of both bounds are equal for
strongly regular graphs.  We refer to \cite{Brouwer1989,Brouwer2012}
for a general discussion of strongly regular graphs.  In this paper a
strongly regular graph $\Gamma$ has parameters $v$, $k$, $\lambda$,
$\mu$, where $v$ denotes the number of vertices of $\Gamma$, $k$
denotes the degree of each vertex of $\Gamma$, $\lambda$ denotes the
size of the common neighborhood of two adjacent vertices, and $\mu$
denotes the size of the common neighborhood of two non-adjacent
vertices.  For two vertices $u$, $v$ in a graph $\Gamma$, we use
$u\sim v$ to denote that $u$ is adjacent to $v$.  The adjacency matrix
$A$ of a strongly regular graph has three eigenvalues, $k$, $e^+$ and
$e^-$, where $k \geq e^+ \geq 0 > e^-$. A strongly regular graph
$\Gamma$ is \emph{primitive} if both $\Gamma$ and $\overline{\Gamma}$
are connected; this implies that $e^+>0$. In this paper we only
consider primitive strongly regular graphs. We denote the eigenspaces
that correspond to $k$, $e^+$, and $e^-$ by $\langle j \rangle$,
$V^+$, and $V^-$ respectively (where $j$ denotes the all ones vector). We denote the multiplicity of $e^+$ by
$m^+$, and the multiplicity of $e^-$ by $m^-$.  Throughout this paper
every graph that is called $\Gamma$ is strongly regular and its
parameters are named as above.

Hoffman's ratio bound and Cvetcovi\'{c}'s inertia bound state
(respectively) that for a primitive strongly regular graph $\Gamma$, a coclique $Y$ in $\Gamma$ satisfies
\begin{align*}
  |Y| \leq \frac{ve^-}{e^- - k}, \qquad |Y| \leq m^-. 
\end{align*}
A coclique of size $\frac{ve^-}{e^- - k}$
is called a \textit{Delsarte coclique}. 
We refer to Section \ref{sec:pre} for further definitions.

It has long been known that the case where both bounds are tight is special;
for example Haemers investigated this in his PhD thesis in 1979 \cite[Th. 2.1.7]{Haemers1979}.
Our work extends an investigation by Haemers and Higman \cite{Haemers1989} for strongly regular graphs
in general, and results by Makhnev and Makhnev for generalized quadrangles of order $(q, q^2-q)$ \cite{Makhnev2003}.
For our purposes, we phrase some of these results in a different way
than the existing literature. For example, we provide
a short translation of \cite[Th. 9.4.1]{Brouwer2012} as Theorem \ref{thm:subsrg}.

Our main results are based on designs derived from strongly
regular graphs. A $2$-$(\tilde{v}, \tilde{k}, \tilde{\lambda})$ design
with intersection numbers $s_1$ and $s_2$ is a set of $\tilde{k}$-sets
$\scrB$ such that
\begin{enumerate}[(a)]
 \item all $b \in \scrB$ satisfy $b \subseteq \{ 1, \ldots, \tilde{v} \}$,
 \item each pair $\{ i, j \} \subseteq \{ 1, \ldots, \tilde{v} \}$ lies in exactly
      $\tilde{\lambda}$ elements of $\scrB$, and
 \item $|b \cap b'| \in \{ s_1, s_2, \tilde{k} \}$
      for all $b, b' \in \scrB$.
\end{enumerate}
If $s_1 \neq s_2$, then $\scrB$ is called a \emph{quasisymmetric design}.
If $s_1 = s_2$, then $\scrB$ is called a \emph{symmetric design}. 
The \emph{replication number} $\tilde{r}$ denotes the number of blocks that contain one
given element and satisfies
\begin{align*}
  \tilde{r}(\tilde{k}-1) = (\tilde{v}-1)\tilde{\lambda}.
\end{align*}
Notice that we have $\tilde{r} = \tilde{k}$ if and only if the design
is symmetric.

Our next result describes how to constuct designs from the
strongly regular graphs that we consider in this paper. We will see
that this result is a reformulation of Theorem 9.4.1 from \cite{Brouwer2012}.

\begin{theorem}\label{thm:quasisdesign}
  Let $\Gamma$ be a primitive strongly regular graph, $Y$ a Delsarte coclique
  of $\Gamma$ and $ve^- = m^-(e^- - k)$.
  Let $\overline{Z}$ denote the set of vertices of $\Gamma$ that are not in $Y$.
  For $z \in \overline{Z}$, let $b_z$ denote 
    $\{ y \in Y: y \sim z \}$.
  Then $\overline{\scrB} := \{ b_z: z \in \overline{Z} \}$
  is a quasisymmetric $2$-$(m^-, -e^-,\mu)$ design with replication number $k$. 
  Two adjacent vertices in $\overline{Z}$ correspond
  to two blocks with intersection size $-(e^+)^2 - e^+ - e^-$, while two non-adjacent
  vertices in $\overline{Z}$ correspond to two blocks with intersection size $-(e^+)^2 - e^-$.
\end{theorem}

We combine this result with existence results for quasisymmetric designs due to Blokhuis and
Calderbank \cite{Blokhuis1992} which rule out equality in the Hoffman
bound for many feasible parameter sets for strongly regular
graphs. We first consider strongly regular graphs
constructed from generalized quadrangles.

A \textit{generalized quadrangle} of order $(s,t)$ consists of a set $P$ of points, a set $L$ of lines, and an incidence structure $I\subseteq P\times L$. We say that a line $l\in L$ contains a point $p$ if $(p,l)\in I$. The incidence structure must satisfy the following:
\begin{enumerate}[(a)]
\item
for each $p\in P$, there are exactly $t+1$ lines in $L$ that contain $p$,
\item
each $l\in L$ contains exactly $s+1$ points in $P$,
\item
if a point $p$ is not on a line $l$, then there is a unique point $p'$ and line $l'$ such that $l'$ contains $p$ and $p'$, and $l$ contains $p'$.
  \end{enumerate}
A generalized quadrangle induces a strongly regular graph $\Gamma$,
called the \textit{point graph}, by setting the vertices of $\Gamma$ to
be the points $P$, and setting $u\sim v$ if and only if there is a line containing both $u$ and $v$. 
A Delsarte coclique of the point graph of a generalized quadrangle 
is traditionally called an \textit{ovoid}. We refer to \cite[Chapter 1]{Payne2009} for details.
Whether certain generalized quadrangles possess an ovoid is a long-standing open question
which has attracted various researchers, see \cite[Sections 1.8 and 3.4]{Payne2009}.
In particular, we show that for infinitely many choices of $q$, generalized quadrangles
of order $(q,q^2-q)$ do not possess ovoids.

Besides the non-existence proofs for some Delsarte cocliques, our main result is the following.
\begin{theorem}\label{thm:symmdesign}
  Let $\Gamma$ be a primitive strongly regular graph, let $Y, Z$ be different Delsarte cocliques
  of $\Gamma$ and $ve^- = m^-(e^- - k)$.
  For a vertex $z$ in $Z \setminus Y$ let $b_z$ denote $\{ y \in Y: y \sim z \}$.
  Then the following hold.
  \begin{enumerate}[(a)]
   \item The set $\scrB := \{ b_z: z \in Z \setminus Y \}$
  is a symmetric $2$-$\left( \frac{(e^+)^2 - (e^-)^{2}}{(e^+)^{2} +
      (e^-)}, -e^-, -(e^+)^2 - e^- \right)$ 
  design.
  \item $|Y \cap Z| = m^- - \frac{(e^+)^2 - (e^-)^2}{(e^+)^2 + e^-} = \frac{(e^-+1)e^+}{(e^+)^2 + e^-}$.
  \item Every block of the quasisymmetric design of Theorem \ref{thm:quasisdesign} corresponding to a vertex that is not in $Z$ contains exactly $e^+$ elements of $Y \cap Z$.
  \item The graph $\Gamma$ contains at most $m^-+1$ Delsarte cocliques. 
  Equality in this bound implies that a symmetric 
  $2$-$\left (\frac{(e^+)^2+ e^+e^-+e^+ - (e^-)^2}{(e^+)^2 + e^-}, \frac{(e^-+1)e^+}{(e^+)^2 + e^-}, \frac{-(e^+)^2 + e^+}{(e^+)^2 + e^-}\right)$ design exists
  and that every vertex of $\Gamma$ lies in exactly $1+\frac{(e^-+1)e^+}{(e^+)^2 + e^-}$ Delsarte cocliques.
  \item The common intersection of three Delsarte cocliques in $\Gamma$ has size $\frac{-(e^+)^2 + e^+}{(e^+)^2 + e^-}$.
  \item If $e^+$ and $e^-$ are coprime and $e^+ > 1$, then $\Gamma$ contains at most two Delsarte cocliques.
  \end{enumerate}
\end{theorem}

The main motivation for Theorem \ref{thm:symmdesign} is to bound the number of Delsarte cocliques in a strongly regular graph.
For most feasible parameter sets, the intersection numbers above are not integers. As an example, for a strongly regular graph with parameters
$v=287$, $k=126$, $\lambda=45$, and $\mu=63$, the intersection of three cocliques given by Part (e) is not an integer, therefore such a graph
can have at most two Delsarte cocliques. In cases like this where the number of Delsarte cocliques is bounded, if a graph exists and possesses a single Delsarte coclique then it must be very asymmetric. 

Notice that a strongly regular graph that has exactly one or two Delsarte cocliques
is not unusual. It can be easily verified that some strongly regular graphs with parameters 
$(45, 32, 22, 24)$ have this property.%
\footnote{See \url{http://www.maths.gla.ac.uk/~es/srgraphs.php} for a list of these strongly regular graphs.}

Theorem \ref{thm:symmdesign} also gives a straight-forward route of constructing such graphs
by extending a symmetric design with parameters as in Theorem \ref{thm:symmdesign}
to a quasisymmetric design as in Theorem \ref{thm:quasisdesign}.

This paper is organized as follows.
After proving Theorem \ref{thm:symmdesign}, we apply it with Theorem \ref{thm:quasisdesign}
to various infinite families of feasible parameter sets. We conclude by investigating all
feasible parameter sets for graphs with up to 1300 vertices by going through Brouwer's
database of strongly regular graphs\footnote{\url{https://www.win.tue.nl/~aeb/graphs/srg/srgtab.html}}.

\section{Preliminaries}\label{sec:pre}

\subsection{Block Designs}
For the following results, let $\scrB$ be a quasisymmetric $2$-$(\tilde{v}, \tilde{k}, \tilde{\lambda})$ design with intersection 
numbers $s_1$ and $s_2$.

\begin{theorem}[{Calderbank \cite[Theorem A]{Calderbank1987}}]\label{thm:q_eq_2_quasisymm}
  If $s_1 \equiv s_2 \mod{2}$ and $\tilde{r} \not\equiv \tilde{\lambda} \mod{4}$, then
  $\tilde{v} \equiv \pm1 \mod{8}$.
\end{theorem}

\begin{theorem}[{Blokhuis and Calderbank \cite[Theorem 4.3]{Blokhuis1992}}]\label{thm:q_neq_2_p_only_quasisymm}
  If $s_1 \equiv s_2 \equiv s \mod{p}$ for some odd prime number $p$ and $\tilde{r} \not\equiv \tilde{\lambda} \mod{p^2}$,
  then one of the following occurs:
  \begin{enumerate}[(a)]
   \item $\tilde{v}$ is even, $\tilde{v} \equiv s \equiv 0 \mod{p}$ and $(-1)^{\tilde{v}/2}$ is a square modulo $p$,
   \item $\tilde{v}$ is even, $\tilde{v} \not\equiv s \not\equiv 0 \mod{p}$, 
	  $\tilde{\lambda} \equiv 0 \mod{p}$ and $(-1)^{\frac{\tilde{v}+2}{2}}\tilde{k} (\tilde{v}-\tilde{k})$ is a square modulo $p$,
   \item $\tilde{v}$ is odd, $\tilde{v} \not\equiv s \equiv 0 \mod{p}$, $\tilde{\lambda} \equiv 0 \mod{p}$, 
	  and $-\tilde{v} (-1)^{\frac{\tilde{v}+1}{2}}$ is a square modulo $p$, or
   \item $\tilde{v}$ is odd, $\tilde{v} \equiv s \not\equiv 0 \mod{p}$,  
	  and $-s (-1)^{\frac{\tilde{v}+1}{2}}$ is a square modulo $p$.
  \end{enumerate}
\end{theorem}

\begin{theorem}[{Blokhuis and Calderbank \cite[Theorem 5.1]{Blokhuis1992}}]\label{thm:q_neq_2_quasisymm}
  Let $p$ be an odd prime number and let $e$ be an odd positive integer. For an integer $z$ define
  $\psi(z) = \max\{ \ell: p^\ell \text{ divides } z, \ell \leq e \}$.
  If $s_1 \equiv s_2 \equiv s \mod{p^e}$, $\tilde{r} \not\equiv \tilde{\lambda} \mod{p^{e+1}}$ and $\tilde{v}$ is odd, then one of the following occurs:
  \begin{enumerate}[(a)]
   \item $\psi(s)$ is odd and $(-1)^{\frac{\tilde{v}-1}{2}} \tau$ is a square modulo $p$,
      where $\tilde{v} - s = p^{\psi(v-s)} \tau$, or
   \item $\psi(s)$ is even and $(-1)^{\frac{\tilde{v}-1}{2}} \sigma$ is a square modulo $p$,
      where $s = p^{\psi(s)} \sigma$.
  \end{enumerate}
\end{theorem}

The following lemma is surely known, but we do not know a reference. 
We include a short proof for the sake of completeness.

\begin{lemma}\label{lem:char_sym_design}
  Let $\scrB$ be a set of $\tilde{k}$-sets of $\{ 1, \ldots, \tilde{v} \}$ with
  $|\scrB| = \tilde{v}$ and suppose that there is a constant $s$ such that 
  $|b \cap b'| \in \{ s, \tilde{k} \}$ for all $b, b' \in \scrB$. Then $\scrB$
  is a symmetric $2$-$(\tilde{v}, \tilde{k}, s)$ design if and only 
  if $\tilde{k}(\tilde{k}-1) = s(\tilde{v}-1)$.
\end{lemma}
\begin{proof}
  We verify that (i) each element of $M := \{ 1, \ldots, \tilde{v} \}$ lies on $\tilde{k}$
  elements of $\scrB$ and that (ii) each pair of $M$ lies on 
  $s$ elements of $\scrB$.
  
  For (i) let $\ell_i$ denote the number of elements in $M$ that lies in exactly 
  $i$ elements of $\scrB$. Standard counting arguments show that  
  \begin{align*}
    &\sum \ell_i = \tilde{v}, && \sum i \ell_i = \tilde{v}\tilde{k}, && \sum i(i-1)\ell_i = \tilde{v}(\tilde{v}-1) s.
  \end{align*}
  Thus,  
  \begin{align}
    0 \leq \sum (\tilde{k}-i)^2 \ell_i = \tilde{v}((\tilde{v}-1)s + \tilde{k}) - \tilde{v}\tilde{k}^2.\label{eq:variance_cnt}
  \end{align}
  Hence, each element of $M$ lies in exactly $\tilde{k}$ elements of $\scrB$ if and
  only if we have equality in \eqref{eq:variance_cnt}, which occurs if and only if $\tilde{k}(\tilde{k}-1) = (\tilde{v}-1)s$.
  
  For (ii) let $r_i$ denote the number of ordered pairs of elements in $M$ that lie in $i$ elements of $\scrB$.
  Notice that we can assume $\tilde{k}(\tilde{k}-1) = (\tilde{v}-1)s$.
  As before we obtain
  \begin{align*}
    &\sum r_i = \tilde{v}(\tilde{v}-1), \\
    &\sum i r_i = \tilde{v}\tilde{k}(\tilde{k}-1) = \tilde{v}(\tilde{v}-1)s,\\
    & \sum i(i-1)r_i = \tilde{v}(\tilde{v}-1) s (s-1).
  \end{align*}
  From this we obtain $\sum (s-i)^2 r_i = 0$. Hence, each pair of $M$ lies in 
  exactly $s$ elements of $\scrB$.
\end{proof}

\subsection{Strongly Regular Graphs}
For a strongly regular graph $\Gamma$, if $\mu > 0$, then the following equations are well-known \cite[Theorem 1.3.1]{Brouwer1989}:
\begin{align}
  &e^+e^- = \mu - k, && v = \frac{(k-e^+)(k-e^-)}{\mu}, && e^++e^- = \lambda - \mu,\label{eq:srg_ev_identities1}\\
  &m^+ = \frac{(e^-+1)k(k-e^-)}{\mu (e^- - e^+)}, && m^- = v - 1 - m^+.\label{eq:srg_ev_identities2}
\end{align}

We are only considering graphs with $ve^- = m^-(e^--k)$. Together with \eqref{eq:srg_ev_identities1}
and \eqref{eq:srg_ev_identities2}, we have
\begin{align}
 & k = \frac{(e^-)^2 - e^-e^+}{e^++1}, && \mu = \frac{e^-(e^+)^2 +
   (e^-)^2}{e^++1}, &&  m^- = \frac{(e^+)^2+e^- e^++e^+ -
    (e^-)^{2}}{(e^+)^{2}+e^-},\label{eq:srg_ev_identities3}
\end{align}
\begin{align}
 &v = \frac{(2e^+-e^-+1)((e^+)^2+e^-e^++e^+-(e^-)^2)}{(e^++1)((e^+)^2+e^-)}.\label{eq:srg_ev_identities4}
\end{align}

\begin{theorem}[{\cite[Proposition 1.3.2]{Brouwer1989}}]\label{thm:hoffmans_bnd}
  Let $\Gamma$ be a strongly regular graph. If $Y$ is a coclique of $\Gamma$, then
  the following statements are equivalent.
  \begin{enumerate}[(a)]
   \item $|Y|(e^- - k) = ve^-$.
   \item Every vertex not in $Y$ has exactly $-e^-$ neighbors in $Y$.
   \item The characteristic vector $\chi$ of $Y$ lies in $\langle j \rangle + V^-$.
  \end{enumerate}
\end{theorem}

Note that the coclique $Y$ in Theorem \ref{thm:hoffmans_bnd} is a Delsarte coclique.

\begin{theorem}[{Theorem 9.4.1 \cite{Brouwer2012}}]\label{thm:subsrg}
  Let $\Gamma$ be a primitive strongly regular graph with 
  $ve^- = m^-(e^- - k)$ and let $Y$ be a Delsarte coclique. Then the subgraph $\Gamma'$ of $\Gamma$ induced by the vertices
  not in $Y$ is strongly regular with degree $k'=k+e^-$, positive eigenvalues $k'$ and $e^+$,
   and negative eigenvalue $e^++e^-$.
\end{theorem}

Theorem~\ref{thm:quasisdesign} is a restatement of this result. To see
this, let $\overline{\scrB}$ be the quasisymmetric $2$-$(\tilde{v}, \tilde{k}, \tilde{\lambda})$ design
defined in Theorem~\ref{thm:quasisdesign} constructed from a strongly
regular graph $\Gamma$ with parameters $(v,k,\lambda, \mu)$.
By the definition of $\overline{\scrB}$, $\tilde{v} = |Y| = m^-$. 
By Theorem \ref{thm:hoffmans_bnd}, $\tilde{k} = -e^-$.
By the definition of $\Gamma$, $\tilde{\lambda} = \mu$ and the
replication number is $k$.

Let $\Gamma'$ be the subgraph of $\Gamma$ induced by the vertices not
in a given Delsarte coclique $Y$. Clearly the degree of $\Gamma'$ is $k'=k+e^-$.
Applying Equation \eqref{eq:srg_ev_identities1} with these
eigenvalues, we obtain that the size $\mu'$ 
of the common neighborhood of two non-adjacent vertices in $\Gamma'$ satisfies
\begin{align*}
  &\mu' = e^+(e^-+e^+) + k' = e^+(e^-+e^+) + k + e^-.
\end{align*}
Hence, by \eqref{eq:srg_ev_identities1}, one of the intersection sizes
for the blocks in the quasisymmetric design is
\begin{align*}
  &\mu - \mu' = (e^+e^- + k) - (e^+(e^-+e^+) + k + e^-) = -(e^+)^2 - e^-.
\end{align*}

Similarly, by \eqref{eq:srg_ev_identities1}, the size $\lambda'$ of
the common neighborhood of two adjacent vertices in $\Gamma'$
satisfies
\begin{align*}
  &\lambda' = e^++(e^+ + e^-) + \mu' = 2e^++e^- + \mu'.
\end{align*}
Hence, by \eqref{eq:srg_ev_identities1}, the other intersection size is
\begin{align*}
  &\lambda - \lambda' = e^++e^- - (2e^++e^-) + (\mu-\mu') = -(e^+)^2 - e^+ - e^-.
\end{align*}

\section{Proof of Theorem \ref{thm:symmdesign}}

Parts (a) and (b) of Theorem \ref{thm:symmdesign} are surely known,
but we have only found them for a special case in literature due to
Makhnev et al. \cite[Lemma 2]{Makhnev2003}, which limits itself to
strongly regular graphs with the same parameters as the point graph of
a generalized quadrangle of order $(q, q^2-q)$.  Part (c) may be
known, but it was only observed for the special case mentioned before
\cite[Lemma 2]{Makhnev2003}. Part (d) is new. Part (e) and Part (f)
are straight-forward generalizations of \cite[Proposition
2]{Makhnev2003}.

\begin{proof}[Proof of Theorem \ref{thm:symmdesign}]
  Recall that $\scrB := \{ b_z: z \in Z \setminus Y \}$ where $b_z =
  \{y \in Y : y \sim z\}$. Since $Y$ and
  $Z$ are both Delsarte cocliques, we can switch the roles of $Y$ and
  $Z$ in the definition of $\overline{\scrB}$ in
  Theorem~\ref{thm:quasisdesign}.  By Theorem \ref{thm:quasisdesign},
  each set in $\scrB$ has size $-e^-$ and any two distinct sets have
  intersection size $-(e^+)^2 - e^-$.

  Apply standard double-counting arguments to triples $(b, p, p')$,
  where $p, p' \in Y \setminus Z$ and $b \in Z \setminus Y$, with both
  $p$ and $p'$ adjacent to $b$ in $\Gamma$, to obtain
  \begin{align*}
    |\scrB|= \frac{(e^+)^2 - (e^-)^2}{(e^+)^2 + e^-}.
  \end{align*}
Then by Lemma~\ref{lem:char_sym_design} $\scrB$ is a symmetric $2$-$\left(\frac{(e^+)^2 - (e^-)^2}{(e^+)^2 + e^-},
  -e^-, -(e^+)^2 - e^-\right)$. 

  Thus Part (a) holds. Part (b) is implied by (a) and Equation \eqref{eq:srg_ev_identities3}.
  
  Part (c) follows from Theorem \ref{thm:hoffmans_bnd} if we can show that there is equality in Hoffman's bound for the coclique $Z\setminus Y$ in the graph
  $\Gamma'$ induced by the vertices of $\Gamma$ not in $Y$. Using the parameters from Theorem \ref{thm:subsrg}, we need to show that
 \begin{align*}
    \frac{(v-m^-)(e^+ + e^-)}{(e^+ + e^-) - (k + e^-)} = \frac{(e^+)^2 - (e^-)^2}{(e^+)^2 + e^-}
  \end{align*}
  holds. Using $ve^- = m^-(e^- - k)$ and \eqref{eq:srg_ev_identities1}, this is equivalent to
  \begin{align*}
    \mu (e^+ - e^-) = -k((e^+)^2 + e^-).
  \end{align*}
  By \eqref{eq:srg_ev_identities3}, this is true.
  Thus, we can apply Part (b) of Theorem \ref{thm:hoffmans_bnd} and the eigenvalues given by Theorem \ref{thm:subsrg}
  to show that each vertex in $\Gamma'$ that is not in $Z$ is adjacent to exactly $- e^+ - e^-$ vertices in $Z\setminus Y$. 
  Since the vertices in $\Gamma'$ that are not in $Z$ are adjacent to $-e^-$ vertices in $Z$, they must be adjacent to $e^+$ vertices in $Z\cap Y$.
  
  For the bound in (d), suppose that we have $c$ Delsarte cocliques $Y_1, \ldots, Y_c$. Consider the $c \times v$ matrix $M$
  whose rows are the characteristic vectors of the cocliques $Y_i$. By Theorem \ref{thm:hoffmans_bnd}, the row span of $M$ 
  lies in $\langle j \rangle + V^-$, so $\rank(M) \leq 1+m^-$.
  By Part (b),
  \begin{align*}
    ((e^+)^2 + e^-)MM^T = ((e^+)^2 - (e^-)^2) I + (e^-e^+ + e^+) J,
  \end{align*}
  where $I$ is the identity matrix and $J$ is the all-ones matrix.  Since
  $(e^+)^2-(e^-)^2 > e^-e^++e^+$ the eigenvalues of this matrix are
  strictly positive.  Hence, $\rank(MM^T) = c$.  As
  $\rank(MM^T) \leq \rank(M) \leq 1+m^-$, we obtain the desired bound.
  
  For Part (e) and Part (f), let $Y_1$, $Y_2$, and $Y_3$ be three different Delsarte cocliques. Let $S$ denote
  $Y_1 \cap Y_2 \cap Y_3$. Let $\beta = |S|$. Double counting the number of edges between 
  $(Y_1 \cap Y_2) \setminus S$ and $Y_3 \setminus (Y_1 \cup Y_2)$, we obtain
  \begin{align*}
    \left( \frac{(e^-+1)e^+}{(e^+)^2 + e^-} - \beta \right) (-e^-) 
  =  \left( m^- - 2 \frac{(e^-+1)e^+}{(e^+)^2 + e^-} + \beta\right) (e^+).
  \end{align*}
  From Equation \eqref{eq:srg_ev_identities3}, we obtain
  \begin{align}
    ((e^+)^2 + e^-)(e^--e^+)\beta = e^+ ((e^+)^2 - e^-e^+ + e^- - e^+). \label{eq:cocl_int_size}
  \end{align}
  This implies (e).
  If $e^-$ and $e^+$ are coprime, then this implies $\beta \equiv 0 \mod{e^+}$. 
  Hence, $\beta = 0$ or $\beta =\gamma e^+ $ for some positive integer $\gamma$. If $\beta = 0$, then Equation \eqref{eq:cocl_int_size} gives
  \begin{align*}
   0 = (e^+)^2 - e^-e^+ + e^- - e^+ = (e^+-e^-)(e^+-1).
  \end{align*}
  As the first factor is always positive, we obtain $e^+=1$.
  If $\beta = \gamma e^+ $, then \eqref{eq:cocl_int_size} gives
  \begin{align*}
   0 =  \gamma (e^+)^2 + \gamma e^- + e^+ - 1 < \gamma( (e^+)^2 +e^- +e^+).
  \end{align*}
  The number $-(e^+)^2 - e^+ - e^-$ is one of our intersection numbers, so it is nonnegative.
  Hence, the right hand side of the above equation is nonpositive and therefore 
$\gamma(e^+)^2 - e^+- e^-$ is negative.
  Thus, the case $\beta = \gamma e^+$ cannot occur and therefore Part (f) holds.
  
  For the second and third part of (d), consider $Y_1 \cap Y_2, Y_1 \cap Y_3, Y_1 \cap Y_4, \ldots$
  as a family of $\tilde{k}$-subsets of $Y_1$. By (b) and (e), we can apply Lemma \ref{lem:char_sym_design}
  with 
  \begin{align*}
    &\tilde{v} = m^-, && \tilde{k} = \frac{(e^-+1)e^+}{(e^+)^2 + e^-}, && s = \frac{-(e^+)^2 + e^+}{(e^+)^2 + e^-}.
  \end{align*}
  Here, using Equation \eqref{eq:srg_ev_identities3}, the identity $\tilde{k}(\tilde{k}-1) = (\tilde{v}-1)s$ is 
  easily verified. Due to this construction we know that every vertex of $\Gamma$ lies in $0$ or
  $1+\frac{(e^-+1)e^+}{(e^+)^2 + e^-}$ Delsarte cocliques. A Delsarte coclique contains
  $m^-$ elements. Using Equation \eqref{eq:srg_ev_identities3} and Equation \eqref{eq:srg_ev_identities4}, we see
  that $v (1 + \frac{(e^-+1)e^+}{(e^+)^2 + e^-}) = m^-(m^-+1)$. This concludes the proof of (d).
\end{proof}

\section{Known Examples With Delsarte Cocliques}

We start by applying Theorems \ref{thm:quasisdesign} and \ref{thm:symmdesign} to the only known examples for such graphs,
namely the complements of triangular graphs and the $M_{22}$ graph on $77$ vertices.

\begin{ex}
  The complements of the triangular graphs $T(n)$ can be defined in the following way: the $2$-subsets 
  of $\{ 1, \ldots, n \}$ are the vertices of the graph and two vertices are adjacent
  if their intersection is empty. It is well-known that this is a graph with parameters
  \begin{align*}
    &v = \binom{n}{2}, && m^- = n-1,\\
    &e^+ = 1, && e^- = -n + 3.
  \end{align*}
  For $n > 4$, from Theorem \ref{thm:quasisdesign} it is easy to verify that the largest independent sets in this graph correspond 
  to the set of all $2$-sets that contain a fixed element.\footnote{This is a 
  very special case of the famous Erd\H{o}s-Ko-Rado theorem.}
  
 By Theorem \ref{thm:symmdesign},
   these independent sets pairwise intersect in exactly $1$ element.
  Notice that this is no longer the case when $n=4$. For example, 
  $\{ \{ 1, 2 \}, \{ 1, 3 \}, \{ 1, 4 \} \}$ and $\{ \{ 1, 2 \}, \{ 1, 3 \}, \{ 2, 3 \} \}$
  are independent sets that share two elements. It is easy to check that the common intersection
  of three Delsarte cocliques is $0$ as claimed by Theorem \ref{thm:symmdesign}.
  
  Notice that for $n=5$ the corresponding graph is the Petersen graph; for $n=6$ the 
  corresponding graph is the point graph of the unique generalized quadrangle of order $(2, 2)$;
  for $n=7$ the corresponding graph is Taylor's $2$-graph with $e^+=1$.
\end{ex}

\begin{ex}
  The $M_{22}$ graph has parameters
  \begin{align*}
    &v = 77, && m^- = 21,\\
    &e^+ = 2, && e^- = -6.
  \end{align*}
  It is well-known and can be easily checked that this graph possesses $22$ cocliques of size $21$ which 
  pairwise intersect in $5$ elements, while the common intersection of three 
  Delsarte cocliques is $1$. This is also implied by Theorem \ref{thm:symmdesign}.
  As we have equality in Theorem \ref{thm:symmdesign} (d), we can identify the 
  intersections $Y_1 \cap Y_2, \ldots, Y_1 \cap Y_{21}$ of the Delsarte cocliques 
  with a $2$-$( 20, 5, 1 )$ design. That is the unique projective plane of order $4$.
  The design from Theorem \ref{thm:symmdesign} (a) is a $2$-$(21, 6, 2)$ design, so
  a biplane of order $4$. It is well-known that there are three such biplanes, but from $M_{22}$ we only obtain
  the unique biplane with an automorphism group of order $11520$.
\end{ex}

\section{Generalized Quadrangles of Order \texorpdfstring{$(q, q^2-q)$}{(q, q\textasciicircum2-q)}}\label{sec:gq}

Although the results in this section are valid for all strongly
regular graphs having the parameters listed below, we state the
results in terms of generalized quadrangles, as there has been great
interest in the existence of Delsarte cocliques in the point graphs of
generalized quadrangles.  Recall that for a generalized quadrangle, a
Delsarte coclique is called an \emph{ovoid}.  It is known that
generalized quadrangles of order $(q, q')$ with $q' > q^2-q$ do not
possess ovoids, while it is an open question whether generalized
quadrangles of order $(q, q^2-q)$ possess ovoids \cite[Section
1.8]{Payne2009}.  We will rule out the existence of ovoids for various
parameters $q$.  Except for $q \in \{ 2, 3 \}$, the existence of a
generalized quadrangle of order $(q, q^2-q)$ is open, so our results
may only apply to an empty set. See \cite[Chapter 6]{Payne2009} for
the unique existing generalized quadrangle of order $(2, 2)$ and the
non-existence of generalized quadrangles of order $(3, 6)$.  As shown
in \cite[Section 1.15]{Brouwer1989}, the parameters of the point graph
of a generalized quadrangle of order $(q, q^2-q)$ are as follows,
where $q$ is an integer larger than $1$.
\begin{align*}
  &v = (q+1)(q^3-q^2+1), && m^- = q^3-q^2+1,\\
  &e^+ = q-1, && e^- = -q^2+q-1.
\end{align*}

\begin{theorem}\label{thm:gqnonexist}
  A generalized quadrangle of order $(q, q^2-q)$ does not possess an ovoid 
  if one of the following cases occurs:
  \begin{enumerate}[(i)]
   \item $q \equiv 3 \mod{8}$, or
   \item $q = \ell p^e + 1$, where $p$ is a prime with $p \equiv 3 \mod{4}$, 
	  $e$ is odd, and $\ell \equiv 2 \mod{4}$. 
  \end{enumerate}
\end{theorem}
\begin{proof}
  By Theorem \ref{thm:quasisdesign}, the existence of an ovoid is equivalent to
  the existence of a quasisymmetric $2$-$\left( q^3-q^2+1, q^2-q+1,
    q^2-q+1 \right)$ design with
  intersection numbers $\{ 1, q \}$ and replication number $q(q^2-q+1)$.
  
  If $q \equiv 3 \mod{8}$, then 
  \begin{align*}
    &q \equiv 1 \mod{2},\\
    &r - \tilde{\lambda} \equiv (q-1)(q^2-q+1) \equiv 2 \cdot 3 \not\equiv 0 \mod{4},\\
    &\tilde{v} \equiv q^3-q^2+1 \equiv 3 \not\equiv \pm 1 \mod{8}.
  \end{align*}
  Hence, Theorem \ref{thm:q_eq_2_quasisymm} implies non-existence of an ovoid for Case (i).
  
  If $q = \ell p^e + 1$, then
  \begin{align*}
    &q \equiv 1 \mod{p^e},\\ & r - \tilde{\lambda} \equiv (q-1)(q^2-q+1) \equiv \ell p^e \mod{p^{e+1}},\\
    &\tilde{v} \equiv q^2(q-1)+1 \equiv 1 \cdot 2 + 1 \equiv 3 \mod{4},\\ &\psi(s) \equiv 0 \mod{2}.
  \end{align*}
  The condition that $p \equiv 3 \mod{4}$ is equivalent to the statement that $-1$ is not a square modulo $p$. 
  It then follows that $-\sigma(-1)^{\frac{\tilde{v} +1}{2}}$ is not a square modulo $p$. 
  Hence, Theorem \ref{thm:q_neq_2_quasisymm} implies non-existence for Case (ii).
\end{proof}

Theorem \ref{thm:gqnonexist} rules out the existence of an ovoid for $q = 7$, but
not for $q \in \{ 4, 5, 6 \}$. 
If several ovoids exist, then, by Theorem \ref{thm:symmdesign},
they pairwise intersect in $(q-1)^2$ points. This is well-known
 for the unique generalized quadrangle of order $(2, 2)$.
This quadrangle belongs to a family of generalized quadrangles of order $(q,q)$
for which strong intersection conditions between ovoids are known \cite{Ball2006}.

The next open case is the generalized quadrangle of order $(4, 12)$.
By Theorem \ref{thm:symmdesign}, we obtain a symmetric $2$-$(40, 13, 4)$ design.
Many such designs are known \cite{Cepulic1994,Spence1991} (for example, we can
take the $1$-dimensional subspaces of $\bbF_3^4$ as elements and the 
$3$-dimensional subspaces of $\bbF_3^4$ as blocks) but it is an open problem to use such a design
 to construct a generalized quadrangle:

\begin{problem}
  Construct a generalized quadrangle of order $(4, 12)$, starting with a $2$-$(40, 13, 4)$ design.
\end{problem}

We doubt that this is possible due to the following lemma: 
\begin{lemma}\label{thm:gq_ovoids}
  A point-transitive generalized quadrangle of order $(q, q^2-q)$, where $q > 2$, does not possess an ovoid.
\end{lemma}
\begin{proof}
  Recall that an ovoid has size $m^-$.
  Suppose that the generalized quadrangle contains at least one ovoid.
  Then we have at least $v/m^- = q+1$ ovoids due to transitivity.
  By \cite{Makhnev2003} (or Theorem \ref{thm:symmdesign} (f)) a generalized quadrangle of order $(q,q^2-q)$,
   where $q>2$, can have at most two ovoids. This is a contradiction, so the quadrangle possesses no ovoids.
\end{proof}
By Lemma \ref{thm:gq_ovoids}, a generalized quadrangle of order $(4,12)$ would be very asymmetric, as it 
could not be point-transitive.

\section{Taylor's \texorpdfstring{$2$}{2}-graph}\label{sec:taylor}

For the case that $q$ is an odd prime power, we refer to \cite{Spence1992} for a definition of Taylor's $2$-graph for $U(3, q)$.
One strongly regular graph associated with Taylor's $2$-graph for $U(3,q)$ has parameters
\begin{align*}
  & v = (q+1)(q^2-q+1), && m^- = q^2-q+1,\\
  & e^+ = \frac{q-1}{2}, && e^- = -\frac{q^2+1}{2}.
\end{align*}
Again, our results hold for all graphs with the same parameters.

\begin{theorem}\label{thm:taylor_ex_cond}
  A strongly regular graph $\Gamma$ with the parameters 
  $v = (q+1)(q^2-q+1)$, $e^+ = \frac{q-1}{2}$, and $e^- = -\frac{q^2+1}{2}$, where $q > 1$ is odd, does not possess a 
  Delsarte coclique if one of the following occurs:
  \begin{enumerate}[(a)]
   \item $q \equiv 5 \mod{8}$, or
   \item $q = 2 \ell p^{e} + 1$, where $p$ is a prime, $p \equiv 3 \mod{4}$, $e$ is odd, $\ell$ is odd, and $\gcd(\ell, p) =1$.
  \end{enumerate}
  Furthermore, $\Gamma$ possesses at most one Delsarte coclique if $q > 3$.
\end{theorem}
\begin{proof}
  By Theorem \ref{thm:quasisdesign}, the existence of an ovoid is equivalent to
  the existence of a quasisymmetric $2$-$( q^2-q+1, \frac{q^2+1}{2}, \frac{q^3+q^2+q+1}{4} )$ design with
  intersection numbers $\{ \frac{q^2+3}{4}, \frac{(q+1)^2}{4} \}$ and replication number $\frac{q^3+q}{2}$.
  
  If $q \equiv 5 \mod{8}$, then 
  \begin{align*}
    &q \equiv 1 \mod{2},\\
    &\tilde{\lambda} - r \equiv \frac{(q^2+1)(1-q)}{4} \equiv 2 \not\equiv 0 \mod{4},\\
    &\tilde{v} \equiv q^2 - q + 1 \equiv 5 \not\equiv \pm1 \mod{8}.
  \end{align*}
  Hence, Theorem \ref{thm:q_eq_2_quasisymm} implies non-existence of a Delsarte coclique for Case (a).
  
  If $q = 2 \ell p^e +1$, then
  \begin{align*}
    &s := 1 \equiv  \frac{q^2+3}{4} \equiv \frac{(q+1)^2}{4} \mod{p^{e+1}},\\
    &q \equiv 1 \mod{p^e},\\
    &\tilde{\lambda}-r \equiv \frac{-q^3+q^2-q+1}{4} \equiv \ell p^e \not\equiv 0 \mod{p^{e+1}},\\
    &\tilde{v} \equiv q^2 - q + 1 \equiv 1 - 2 \cdot \ell \cdot 3 - 1 + 1 \equiv 3 \mod{4},\\
    &\psi(s) \equiv 0 \mod{2},\\
    &\frac{\tilde{v}-1}{2} =  \frac{q^2 - q}{2}  \equiv 1 \mod{2}.
  \end{align*}
  Again, having $p \equiv 3 \mod{4}$ means that
  $\sigma(-1)^{\frac{\tilde{v} -1}{2}} = -1$ is not a square modulo
  $p$. Therefore Theorem \ref{thm:q_neq_2_quasisymm} implies
  non-existence of a Delsarte coclique for Case (b).
  
We now show that $\Gamma$ has at most one Delsarte coclique. Suppose that $\Gamma$ has Delsarte cocliques $Y$ and $Z$. We find that
  $|Y \cap Z| = \frac{(q-1)^2}{q+1} = q-3 + \frac{4}{q+1}$, by Theorem \ref{thm:symmdesign}. For $e^+$, so $q>3$, this is not an integer.
\end{proof}

\paragraph*{Erratum} An earlier version of this file claimed that this SRG derived from Taylor's two-graph possesses no Delsarte cocliques.

\section{Generalized \texorpdfstring{$M_{22}$}{M22} Graphs}\label{sec:m22}

The parameters of the $M_{22}$ graph are part of the following infinite family, where $q$ is a positive integer.
\begin{align*}
  & v = (q^2+2q-1)(q^2+3q+1), && m^- = (q+1)(q^2+2q-1),\\
  & e^+ = q, && e^- = -q^2-q.
\end{align*}
One noteworthy property of these graphs is that $\lambda = 0$. No such graphs seem to be known for $q > 2$.
For $q=1$ these are the parameters of the Petersen graph.

The smallest open case is $q=3$. Here the symmetric design of Theorem \ref{thm:symmdesign} (a)
has parameters $2$-$(45, 12, 3)$ and many such designs are known. 
In particular all of the designs with a non-trivial automorphism group are classified \cite{Crnkovic2016}.
Using an MIP solver, we verified that none of these designs can be extended to a quasisymmetric $2$-$(56, 12, 9)$ design.
Hence, we conjecture that no graph of the above family with $q=3$ contains a Delsarte coclique.
Therefore, the most promising open case for a construction is $q=4$.
Here one would take a symmetric $2$-$(96, 20, 4)$ design and try to extend it
to a quasisymmetric $2$-$(115, 20, 16)$ design with intersection numbers $0$ and $4$.

Following an idea by Alexander L. Gavrilyuk\footnote{Private communication for the $q=3$ case.},
we have the following lemma about the number of cocliques.

\begin{lemma}\label{lem:exc_266_gen}
  If a strongly regular graph with $v=(q^2+2q-1)(q^2+3q+1)$, $e^+ = q$ and $e^- = -q^2-q$ has 
  $m^-+1$ Delsarte cocliques, then there exists a strongly regular graph with parameters
  $v = q^2(q+3)^2$, $e^+ = q$ and $e^- = -q^2-2q$.
\end{lemma}
\begin{proof}
  Our graph $\Gamma$ has parameters $(v, k, \lambda, \mu) = ((q^2+2q-1)(q^2+3q+1), q^2(q+2), 0, q^2)$.
  Suppose that we have equality in Theorem \ref{thm:symmdesign} (d) which implies that $q^3+3q^2+q$
  cocliques of size $q^3+3q^2+q-1$ form a symmetric $2$-$(q^3+3q^2+q-1, q^2+q-1, q-1)$ design.
  Now we can construct a strongly regular graph $\Gamma'$ with $(v, k, \lambda, \mu) = (q^2(q+3)^2, q^3+3q^2+q, 0, q^2+q)$
  as follows: The vertices of $\Gamma'$ consist of the vertices of $\Gamma$, the $q^3+3q^2+q$ new 
  vertices representing the $q^3+3q^2+q$ Delsarte cocliques,
  and a new vertex $z^*$ representing the set of all $q^3+3q^2+q$ Delsarte cocliques. Adjacency is defined as follows:
  \begin{itemize}
   \item Two vertices of $\Gamma'$ are adjacent if they are adjacent in $\Gamma$.
   \item A vertex of $x \in \Gamma$ and a vertex $z$ representing a Delsarte coclique $Z$ are adjacent if $x \in Z$.
   \item The neighborhood of $z^*$ is exactly the set of all vertices representing the $q^3+3q^2+q$ Delsarte cocliques.
  \end{itemize}
  Using Theorem \ref{thm:symmdesign} (d) it is easy to verify that
  $\Gamma'$ is a strongly regular graph.
\end{proof}

Strongly regular graphs with parameters $(v, k, \lambda, \mu) = (324, 57, 0, 12)$ do not exist (see Gavrilyuk and Makhnev \cite{Gavrilyuk2005}).
Hence, we obtain the following.

\begin{cor}\label{cor:exc_266}
  A strongly regular graph with $v=266$, $e^+ = 3$ and $e^- = -12$ has at most $m^-$
  Delsarte cocliques.
\end{cor}

\paragraph*{Update December 2020.} By now, Corollary \ref{cor:exc_266} is superseded by recent work by Munemasa and Tonchev \cite{Munemasa2020}.
That is, we do not have any Delsarte cocliques.

Besides the triangular graphs, this is the only family of parameters for which
we could not rule out the existence of $m^-+1$ Delsarte cocliques in general.
Among all graphs up to $1300$ vertices, there is one more examples that can have more than two Delsarte
cocliques.

\begin{lemma}\label{lem:exc_1036}
  A strongly regular graph with $v = 1036$, $e^+ = 5$ and $e^- = -45$ has at most
  $m^-$ Delsarte cocliques.
\end{lemma}
\begin{proof}
  Suppose that we have equality in Theorem \ref{thm:symmdesign} (d). Then a symmetric
  $2$-$(111, 11, 1)$ design exists. This is equivalent to the existence of a projective
  plane of order $10$ for which non-existence is known \cite{Lam1989}.
\end{proof}

If the set of $m^-+1$ Delsarte cocliques forms a projective plane, then 
$-(e^+)^2 + e^+ = (e^+)^2 + e^-$. By Equation \eqref{eq:srg_ev_identities4},
the number of vertices is an integer if and only if $e^+ \in \{ 1, 2, 5 \}$.
Hence, we cannot find any other strongly regular graphs where equality in Theorem
\ref{thm:symmdesign} (d) induces a projective plane. Indeed we can say the following.

\begin{theorem}
  Let $\Gamma$ be a primitive strongly regular graph with 
  $ve^- = m^-(e^- - k)$. If $\Gamma$ has $m^-+1$
  Delsarte cocliques, then a $2$-$(\tilde{v}, \tilde{k}, c)$ design exists for some constant $c$ and $e^++1$ divides $2c+4$.
\end{theorem}
\begin{proof}
  By Theorem \ref{thm:symmdesign}, we have
  \begin{align*}
    c = \frac{(1-e^+)e^+}{(e^+)^2 + e^-}.
  \end{align*}
  Rearranging yields
  \begin{align*}
    e^- = \frac{e^+ - (c+1)(e^+)^2}{c}.
  \end{align*}
  By Equation \eqref{eq:srg_ev_identities4}, 
  \begin{align*}
  v = &\frac{(e^+ + 1)\left(((c+1)e^+ + c -2)\left( (c+1)^2((e^+)^2 + 1) + (2c - 1)(c + 1)e^+ + 1\right)\right) + 2c+4}{c^2(e^+ + 1)}.
  \end{align*}
  Dividing the numerator by $e^+ + 1$ leaves a remainder of $2c+4$, and thus $e^+ + 1$ must divide $2c+4$.
\end{proof}

For $c = 2$, the $2$-$(\tilde{v}, \tilde{k}, c)$ design is a biplane. The only non-trivial choice of parameters is 
$e^+ = 3$ and $e^- = -12$, so the design in Theorem \ref{thm:symmdesign} (d) has parameters 
$2$-$(56, 11, 2)$. Biplanes with these parameters were classified by Kaski and 
\"{O}sterg\aa{}rd \cite{Kaski2008}. Hence, the smallest case for which we might be able to use
the result to construct new symmetric designs are triplanes, so $c = 3$. Here 
$(e^+, e^-) = (4, -20)$ and $(e^+, e^-) = (9, -105)$ are the two interesting parameter
sets.

\section{Other Graphs with up to \texorpdfstring{$1300$}{1300} Vertices}

In Table \ref{table:1300} we list strongly regular graphs with up to $1300$ vertices. We do not 
include the complements of triangular graphs or graphs with less than $200$ vertices.
The parameters $s_1$ and $s_2$ are the intersection numbers of the quasi-symmetric 2-design in Theorem \ref{thm:quasisdesign}.
The entry \# gives the maximal number of Delsarte cocliques. Notice that except
for Taylor's $2$-graphs it is not known if strongly regular graphs with the given
parameters exist.

\begin{table}
\begin{center}
\begin{tabular}{rrrrrrrrrrll} \toprule
$v$ & $k$ & $\lambda$ & $\mu$ & $e^+$ & $e^-$ & $m^-$ & $s_1$ & $s_2$ & \# & Reference & Remark\\ \midrule
$245$ & $52$ & $3$ & $13$ & $3$ & $-13$ & $49$ & $1$ & $4$ & $2$ &  Sec. \ref{sec:gq}, $q=4$ &  GQ$(4, 12)$\\
$246$ & $85$ & $20$ & $34$ & $3$ & $-17$ & $41$ & $5$ & $8$ & $0$ &  Th. \ref{thm:q_neq_2_p_only_quasisymm}, $p=3$\\
$261$ & $176$ & $112$ & $132$ & $2$ & $-22$ & $29$ & $16$ & $18$ & $0$ &  \cite{Haemers1989,Tonchev1986} \\
$266$ & $45$ & $0$ & $9$ & $3$ & $-12$ & $56$ & $0$ & $3$ & $0$ &  Cor. \ref{cor:exc_266} et seq. &  Gen. $M_{22}$\\
$287$ & $126$ & $45$ & $63$ & $3$ & $-21$ & $41$ & $9$ & $12$ & $2$ &  Th. \ref{thm:symmdesign} (e)\\
$344$ & $175$ & $78$ & $100$ & $3$ & $-25$ & $43$ & $13$ & $16$ & $0$ &  Sec. \ref{sec:taylor}, $q=7$ &  Taylor's\\
$490$ & $297$ & $168$ & $198$ & $3$ & $-33$ & $49$ & $21$ & $24$ & $2$ &  Th. \ref{thm:symmdesign} (e)\\
$532$ & $156$ & $30$ & $52$ & $4$ & $-26$ & $76$ & $6$ & $10$ & $2$ &  Th. \ref{thm:symmdesign} (e)\\
$568$ & $217$ & $66$ & $93$ & $4$ & $-31$ & $71$ & $11$ & $15$ & $2$ &  Th. \ref{thm:symmdesign} (e)\\
$606$ & $105$ & $4$ & $21$ & $4$ & $-21$ & $101$ & $1$ & $5$ & $2$ &  Sec. \ref{sec:gq}, $q=5$ &  GQ$(5, 20)$\\
$639$ & $288$ & $112$ & $144$ & $4$ & $-36$ & $71$ & $16$ & $20$ & $2$ &  Th. \ref{thm:symmdesign} (e)\\
$667$ & $96$ & $0$ & $16$ & $4$ & $-20$ & $115$ & $0$ & $4$ & $116$ &  Sec. \ref{sec:m22} & Gen. $M_{22}$\\
$672$ & $451$ & $290$ & $328$ & $3$ & $-41$ & $56$ & $29$ & $32$ & $0$ &  Th. \ref{thm:q_neq_2_p_only_quasisymm}, $p=3$\\
$730$ & $369$ & $168$ & $205$ & $4$ & $-41$ & $73$ & $21$ & $25$ & $1$ &  Sec. \ref{sec:taylor}, $q=9$  &  Taylor's\\
$836$ & $460$ & $234$ & $276$ & $4$ & $-46$ & $76$ & $26$ & $30$ & $2$ &  Th. \ref{thm:symmdesign} (e)\\
$1003$ & $300$ & $65$ & $100$ & $5$ & $-40$ & $118$ & $10$ & $15$ & $2$ &  Th. \ref{thm:symmdesign} (e)\\
$1016$ & $259$ & $42$ & $74$ & $5$ & $-37$ & $127$ & $7$ & $12$ & $0$ &  Th. \ref{thm:q_neq_2_p_only_quasisymm}, $p=5$\\
$1017$ & $344$ & $91$ & $129$ & $5$ & $-43$ & $113$ & $13$ & $18$ & $0$ &  Th. \ref{thm:q_neq_2_p_only_quasisymm}, $p=5$\\
$1036$ & $375$ & $110$ & $150$ & $5$ & $-45$ & $111$ & $15$ & $20$ & $111$ &  Lem. \ref{lem:exc_1036}\\
$1080$ & $221$ & $22$ & $51$ & $5$ & $-34$ & $144$ & $4$ & $9$ & $0$ &  Th. \ref{thm:q_neq_2_p_only_quasisymm}, $p=5$\\
$1090$ & $441$ & $152$ & $196$ & $5$ & $-49$ & $109$ & $19$ & $24$ & $2$ &  Th. \ref{thm:symmdesign} (e)\\
$1122$ & $209$ & $16$ & $44$ & $5$ & $-33$ & $153$ & $3$ & $8$ & $0$ &  Th. \ref{thm:q_neq_2_p_only_quasisymm}, $p=5$\\
$1136$ & $855$ & $630$ & $684$ & $3$ & $-57$ & $71$ & $45$ & $48$ & $1$ &  Th. \ref{thm:symmdesign} (b)\\
$1199$ & $550$ & $225$ & $275$ & $5$ & $-55$ & $109$ & $25$ & $30$ & $2$ &  Th. \ref{thm:symmdesign} (e)\\
$1267$ & $186$ & $5$ & $31$ & $5$ & $-31$ & $181$ & $1$ & $6$ & $2$ &   Sec. \ref{sec:gq}, $q=6$ &  GQ$(6, 30)$\\
 \bottomrule
\end{tabular}
\end{center}
 \caption{All strongly regular graphs with $ve^- = m^-(e^--k)$ and $e^+>1$ for $v \in [200, 1300]$.}
 \label{table:1300}
\end{table}

\section*{Acknowledgments} The work in this paper was a joint project of the
Discrete Mathematics Research Group at the University of Regina, attended by
all the authors.
We would like to thank Willem H. Haemers, Koen Thas and Joseph A. Thas for pointing us to various 
references, in particular \cite{Haemers1989} and \cite{Makhnev2003}. 
We would like to thank Alexander L. Gavrilyuk for pointing out the idea for Lemma \ref{lem:exc_266_gen}.

\bibliographystyle{plain}


\begin{thebibliography}{10}

\bibitem{Ball2006}
S.~Ball, P.~Govaerts, and L.~Storme.
\newblock On ovoids of parabolic quadrics.
\newblock {\em Des. Codes Cryptogr.}, 38(1):131--145, 2006.

\bibitem{Blokhuis1992}
A.~Blokhuis and A.~R. Calderbank.
\newblock Quasi-symmetric designs and the {S}mith normal form.
\newblock {\em Des. Codes Cryptogr.}, 2(2):189--206, 1992.

\bibitem{Brouwer1989}
A.~E. Brouwer, A.~M. Cohen, and A.~Neumaier.
\newblock {\em Distance-regular graphs}, volume~18 of {\em Ergebnisse der
  Mathematik und ihrer Grenzgebiete (3)}.
\newblock Springer-Verlag, Berlin, 1989.

\bibitem{Brouwer2012}
A.~E. Brouwer and W.~H. Haemers.
\newblock {\em Spectra of graphs}.
\newblock Universitext. Springer, New York, 2012.

\bibitem{Calderbank1987}
A.~R. Calderbank.
\newblock The application of invariant theory to the existence of
  quasisymmetric designs.
\newblock {\em J. Combin. Theory Ser. A}, 44(1):94--109, 1987.

\bibitem{Cepulic1994}
V.~\'Cepuli\'c.
\newblock On symmetric block designs {$(40,13,4)$} with automorphisms of order
  {$5$}.
\newblock {\em Discrete Math.}, 128(1-3):45--60, 1994.

\bibitem{Crnkovic2016}
D.~Crnkovi\'c, D.~Dumi\v{c}i\'c~Danilovi\'c, and S.~Rukavina.
\newblock Enumeration of symmetric {$(45,12,3)$} designs with nontrivial
  automorphisms.
\newblock {\em J. Algebra Comb. Discrete Struct. Appl.}, 3(3):145--154, 2016.

\bibitem{Gavrilyuk2005}
A.L.~Gavrilyuk, A.A.~Makhnev. 
\newblock On Krein graphs without triangles.
\newblock {\em Doklady Mathematics}, 72:1, 591–594, 2005.

\bibitem{Haemers1979}
W.~H. Haemers.
\newblock {\em Eigenvalue techniques in design and graph theory}.
\newblock PhD thesis, Technische Hogeschool Eindhoven, 1979.

\bibitem{Haemers1989}
W.~H. Haemers and D.~G. Higman.
\newblock Strongly regular graphs with strongly regular decomposition.
\newblock {\em Linear Algebra Appl.}, 114/115:379--398, 1989.

\bibitem{Kaski2008}
P.~Kaski and P.~R.~J. \"Osterg\aa{}rd.
\newblock There are exactly five biplanes with {$k=11$}.
\newblock {\em J. Combin. Des.}, 16(2):117--127, 2008.

\bibitem{Lam1989}
C.~W.~H. Lam, L.~Thiel, and S.~Swiercz.
\newblock The nonexistence of finite projective planes of order {$10$}.
\newblock {\em Canad. J. Math.}, 41(6):1117--1123, 1989.

\bibitem{Makhnev2003}
A.~A. Makhnev, Jr. and A.~A. Makhnev.
\newblock Ovoids and bipartite subgraphs in generalized quadrangles.
\newblock {\em Mat. Zametki}, 73(6):878--885, 2003.

\bibitem{Munemasa2020}
A. Munemasa and V.~D. Tonchev.
\newblock Ternary codes, biplanes, and the nonexistence of some quasisymmetric and quasi‐3 designs.
\newblock {\em J. Combin. Des}, 28(10):745--752, 2020.

\bibitem{Payne2009}
S.~E. Payne and J.~A. Thas.
\newblock {\em Finite generalized quadrangles}.
\newblock EMS Series of Lectures in Mathematics. European Mathematical Society
  (EMS), Z\"urich, second edition, 2009.

\bibitem{Spence1991}
E.~Spence.
\newblock {$(40,13,4)$}-designs derived from strongly regular graphs.
\newblock pages 359--368, 1991.

\bibitem{Spence1992}
E.~Spence.
\newblock Is {T}aylor's graph geometric?
\newblock {\em Discrete Math.}, 106/107:449--454, 1992.

\bibitem{Tonchev1986}
V.~D. Tonchev.
\newblock Quasisymmetric designs and self-dual codes.
\newblock {\em European J. Combin.}, 7(1):67--73, 1986.

\end{thebibliography}

\end{document}